\documentclass[10pt,twoside,reqno,a4paper]{article}
\usepackage[T1]{fontenc}
\usepackage{a4wide}

\usepackage{amsfonts, amsmath, amssymb,latexsym}
\usepackage[activeacute,english]{babel}
\usepackage{epsfig}
\usepackage{graphicx}
\usepackage[font=sf,sc]{caption}
\usepackage{float}
\usepackage{cancel}
\usepackage{color}
\usepackage[utf8]{inputenc}
\usepackage{amsthm}
\usepackage{bm}
\usepackage[all]{xy}
\usepackage{enumerate}
\usepackage{comment}
\usepackage[colorinlistoftodos]{todonotes} 
\usepackage{bbm}

\usepackage[colorlinks=true,linkcolor=blue,urlcolor=blue,citecolor=blue]{hyperref}

\newcommand{\Q}{{\mathbb Q}}
\newcommand{\R}{{\mathbb {R}}}

\renewcommand{\leq}{\leqslant}
\renewcommand{\geq}{\geqslant}

\newcommand{\RR}{\mathbb{R}}

\newtheorem{theorem}{Theorem}
\newtheorem{thm}[theorem]{Theorem}
\theoremstyle{plain}

\newtheorem{cor}[theorem]{Corollary}

\newtheorem{exam}[theorem]{Example}

\newtheorem{lem}[theorem]{Lemma}

\newtheorem{prop}[theorem]{Proposition}

\newtheorem{obs}[theorem]{Observation}

\numberwithin{equation}{section}
\numberwithin{theorem}{section}

\allowdisplaybreaks

\theoremstyle{definition}

\usepackage{enumerate}

\usepackage[
backend=bibtex,
style=numeric,
sorting=nyt,
giveninits=true, 
doi=false,isbn=false,url=false,eprint=false,
]{biblatex}
\DeclareNameAlias{author}{family-given}
\appto{\bibsetup}{\sloppy}
\usepackage{csquotes}
\RequirePackage{dsfont} 
\newcommand{\EE}{\mathbb{E}} 
\newcommand{\PP}{\mathbb{P}} 
\newcommand{\QQ}{\mathbb{Q}} 
\newtheorem{ass}{Assumption}

\title{Change of measure in a Heston-Hawkes stochastic volatility model}

\author{David R. Ba\~{n}os\thanks{Department of Mathematics, University of Oslo, P.O. Box 1053 Blindern, N-0316 Oslo, Norway, Email: \mbox{davidru@math.uio.no}}, 
Salvador Ortiz-Latorre\thanks{Department of Mathematics, University of Oslo, P.O. Box 1053 Blindern, N-0316 Oslo, Norway, Email: \mbox{salvadoo@math.uio.no}}, 
and 
Oriol Zamora Font\thanks {Department of Mathematics, University of Oslo, P.O. Box 1053 Blindern, N-0316 Oslo, Norway, Email: \mbox{oriolz@math.uio.no}}}
\date{\today}

\begin{document}
\maketitle

\begin{abstract}
We consider the stochastic volatility model obtained by adding a compound Hawkes process to the volatility of the well-known Heston model. A Hawkes process is a self-exciting counting process with many applications in mathematical finance, insurance, epidemiology, seismology and other fields. We prove a general result on the existence of a family of equivalent (local) martingale measures. We apply this result to a particular example where the sizes of the jumps are exponentially distributed. 

{\it Keywords:} stochastic volatility, change of measure, risk neutral measure, existence of equivalent martingale measure, non-Markovian model, Hawkes process, volatility with jumps.

{\it AMS classification MSC2020:} 60G55, 60J76, 91G15, 60H10.
\end{abstract}

\section{Introduction}\label{Intro}

Valuation of assets and financial derivatives constitutes one of the core subjects of modern financial mathematics. There have been several approaches to asset pricing all of which can be classified in two larger groups: equilibrium pricing and rational pricing. The latter gives rise to the commonly used methodology of pricing financial instruments by ruling out arbitrage opportunities. As it is well-known, the absence of arbitrage is closely related to the existence of a probability measure, so-called risk neutral measure, hereby denoted by $\Q$. Prices of derivatives are expectations under such measure. This connection is known as the fundamental theorem of asset pricing \cite{DS94}. 

In contrast, markets evolve in time and they do it under the so-called real world measure, hereby denoted by $\PP$. While we can attempt to model market movements under $\PP$, we also need the dynamics under $\Q$ for pricing purposes. If we are only interested in pricing, then modelling under $\Q$ and calibrating is possible by matching market prices to theoretical ones. However, for asset liability management, investors need often to assess their positions under $\PP$. A common risk management practice is to compute the risk exposure as a way to set economic and regulatory capital levels. \cite{Stein16} shows that exposures computed under the risk neutral measure are essentially arbitrary. They depend on the choice of numéraire and can be manipulated by choosing a different numéraire. Even when choosing commonly used numéraires, these exposures can differ by a factor of two or more. Furthermore, a crucial feature when assessing risk exposures is their distribution. While models under $\Q$ may have well-known tractable distributional properties, this need not be the case when doing the passage to the $\PP$-world and vice versa. For instance, the Vasicek model for interest rates is invariant under a restricted family of measure changes. This is a sought property of the Vasicek model, but it does not need to hold true in other models.

It is a common practice in the literature to assume that such measure change exists and set up a model under the risk neutral measure. Nonetheless, such assumption is not innocuous and nonsensical results can occur if it is in fact not satisfied, see e.g. the discussion in \cite{onchanges} and the references therein. For instance, \cite{BS96} and \cite{Ryd99} show examples of models where no equivalent local martingale measure exist. See also \cite{Ryd97} for conditions to check existence of equivalent martingale measures under Markovian models without jumps.

From the modelling perspective, we adopt a stochastic volatility model in a non-Markovian setting with jumps. The reason for this is twofold. The non-Markovianity can be justified by the clustering of volatility, see e.g. \cite{Cont07}. This is a well-known stylized fact of volatility that cannot be captured under Markovianity. Literature on non-Markovian stochastic volatility models is vast. Here we mention some works on fractional volatility models, which appear when using the fractional Brownian motion as a driving noise for the volatility. The main characteristics of such noise are the possibility to model short and long-range dependence due to the fractional decay of its auto-correlation function and the ability to generate trajectories which are Hölder continuous of index different from $1/2$, corresponding to the case of Brownian motion.  Fractional models with long range dependence were first studied in \cite{ComRen98}. Then, in \cite{AloLeVi07}, Malliavin calculus is introduced to study the asymptotics of the implied volatility in general stochastic volatility models, including fractional models with long and short term dependence. The popularity of rough models (short term dependence models) started with the findings of \cite{Gatheral}. Their name stemming from the roughness of the underlying noise in the stochastic evolution of volatility. Several works in this direction have been conducted, see e.g. \cite{Gatheral}, \cite{perfecthedging}, \cite{characteristicrough} to mention a few. None of the aforementioned works discusses change of measure. In fact, it is not clear whether the volatility process in the rough Heston model (see e.g. \cite{characteristicrough} and \cite{perfecthedging}) is strictly positive, which is a necessary condition for a proper change of measure to be possible. 

The rough Heston model extends the classical Heston model to a model with fractional noise. It is moreover proven in \cite{roughHawkes1}, \cite{roughHawkes2}, \cite{roughHawkes3} that the rough Heston model can be obtained as a limit of Heston models where the noise is a pure jump process known as Hawkes process. The latter, justifies again our choice of model with a self-exciting Hawkes factor.

The inclusion of jumps in the volatility may not be so clear at first glance. Often, researchers include jumps in the stock process in order to capture sudden changes of asset returns, see for instance the works of \cite{Andersenetal02}, \cite{Chernovetal00}, \cite{Erakeretal03}. Classical stochastic volatility models without jumps, such as the Heston model, have the property that returns conditional on volatility are normally distributed. This property fails to explain many features of asset price behaviour. By adding jumps in the returns one can model sudden changes such as economic crashes as well as having more freedom in the modelling of distributions of asset returns. Nevertheless, jumps in returns are transient, in other words, a jump in returns today has no impact on the future distribution of returns \cite{Erakeretal03}. Hence, Markov models for jumps in the asset dynamics such as Poisson process are often used. On the other hand, volatility is highly persistent. If the dynamics is driven by a Brownian motion then volatility can only increase gradually by a sequence of small normally distributed increments. Self-exciting jumps in the volatility provide a rapidly moving persistent factor. This justifies our use of the Hawkes component in the volatility.

Several works support the presence of jumps in the diffusive volatility. \cite{Bates00}, \cite{Duffieetal03} and \cite{Pan02} provide evidence for the presence of positive jumps in volatility. For example, the authors in \cite[Section 5.4]{Pan02} conduct an empirical study on the higher moments of the volatility process seeking evidence of jumps in volatility, as firstly conjectured by \cite{Bates00}. Their findings indicate the possibility of jumps (with positive mean jump size) in the stochastic-volatility process or at least fatter tailed innovations in the volatility process. This justifies our choice of positive self-exciting jumps in our volatility, which we model by compounding a Hawkes process with independent jump sizes. The positivity of jumps also allows us to keep volatility from hitting zero with probability one but, on the other hand, the self-exciting property of the Hawkes process may cause the process to explode. For this reason we assume a stability property to prevent explosion, see \cite[Section 3.1.1]{article1}.

In this work we look at a Heston-type stochastic volatility model with correlated Brownian motions and add a jump part to the volatility process, namely, a compound Hawkes process. We show that a passage from $\PP$ to $\Q$ and vice versa is possible for a rich enough family of probability measures. It is worth noting that our model has three non-tradable noises and one tradable asset, being thus incomplete. This gives rise to non-unique choices of measure change. The proof of existence of equivalent martingale measures for the classical Heston model is conducted in \cite{onchanges}. 

The paper is organized as follows: in Section \ref{Frame} we present our stochastic volatility model with correlated Brownian noises and a compound Hawkes component in the volatility. We also present in this section some useful properties on the existence and positivity of the volatility process. In Section \ref{sec:change}  we prove the main results in this paper. First, we study the integrability of the exponential of the integrated variance and then we prove the existence of (local) martingale measures in our model. Finally, in Appendix \ref{sec: appendix}, we prove some technical lemmas on the existence of solutions of ODEs that appear in the proof of the main results.
 
\section{Stochastic volatility model}\label{Frame}

Let $T\in \R$ , $T>0$ be a fixed time horizon. On a complete probability space $(\Omega,\mathcal{A},\PP)$, we consider a two-dimensional standard Brownian motion $(B,W)=\left\{\left(B_t,W_t\right), t\in[0,T]\right\}$ and its minimally augmented filtration $\mathcal{F}^{(B,W)}=\{\mathcal{F}^{(B,W)}_t, t\in[0,T]\}$. On $(\Omega,\mathcal{A},\PP)$, we also consider a Hawkes process $N=\{N_t, t\in[0,T]\}$ with stochastic intensity given by
\begin{align*}
    \lambda_t=\lambda_0+\alpha\int_0^t e^{-\beta(t-s)}dN_s,
\end{align*}
or, equivalently, 
\begin{align*}
    d\lambda_t=-\beta(\lambda_t-\lambda_0)dt+\alpha dN_t,
\end{align*}
where $\lambda_0>0$ is the initial intensity, $\beta>0$ is the speed of mean reversion and $0<\alpha<\beta$ is the self-exciting factor. Note that the stability condition
\begin{align*}
    \alpha\int_0^\infty e^{-\beta s}ds=\frac{\alpha}{\beta}<1,
\end{align*}
holds. See \cite[Section 2]{Hawkesfinance} and \cite[Section 3.1.1]{article1} for the definition of $N$. 
We assume that $(B,W), N$ and $\{J_i\}_{i\geq 1}$ are independent of each other. 
Then, we consider a sequence of i.i.d., strictly positive and integrable random variables $\{J_i\}_{i\geq 1}$ and the compound Hawkes process $L=\{L_t, t\in[0,T]\}$ given by
\begin{align*}
    L_t=\sum_{i=1}^{N_t}J_i.
\end{align*}   
We write $\mathcal{F}^L=\left\{\mathcal{F}^L_t, t\in[0,T]\right\}$ for the minimally augmented filtration generated by $L$ and
\begin{align*}
    \mathcal{F}=\{\mathcal{F}_t=\mathcal{F}_t^{(B,W)}\vee\mathcal{F}_t^L, t\in[0,T]\},
\end{align*}
for the joint filtration. We assume that $\mathcal{A}=\mathcal{F}_T$ and we will work with $\mathcal{F}$. Since $(B,W)$ and $L$ are independent processes, $(B,W)$ is also a two-dimensional $(\mathcal{F},\PP)$-Brownian motion. 

Finally, with all these ingredients, we introduce our stochastic volatility model. We assume that the interest rate is deterministic and constant equal to $r$, but a non-constant interest rate can easily be fitted into this framework. The stock price $S=\{S_t, t\in[0,T]\}$ and its variance $v=\{v_t, t\in[0,T]\}$ are given by
\begin{align}\label{2}
    \frac{dS_t}{S_t} & =\mu_tdt+\sqrt{v_t}\left(\sqrt{1-\rho^2} dB_t+\rho dW_t\right), \\
    \label{21} dv_t & =-\kappa\left(v_t-\Bar{v}\right)dt+\sigma\sqrt{v_t}dW_t+\eta dL_t,
\end{align}
where $S_0>0$ is the initial price of the stock, $\mu: [0,T] \rightarrow \R$ is a measurable and bounded function, $\rho\in(-1,1)$ is the correlation factor, $v_0>0$ is the initial value of the variance, $\kappa>0$ is the variance's mean reversion speed, $\Bar{v}>0$ is the long-term variance, $\sigma>0$ is the volatility of the variance and $\eta>0$ is a scaling factor. We assume that the Feller condition $2\kappa\Bar{v}\geq\sigma^2$ is satisfied, see \cite[Proposition 1.2.15]{AlfonsiAurélien2015ADaR}.

Note that our stochastic volatility model is the well-known Heston model but adding a compound Hawkes process in the variance process. The procedure to prove strong existence and pathwise uniqueness of the SDE in \eqref{21} is essentially written in \cite[Chapter V.10, Theorem 57]{Pro1992}. Informally, between jump times, the SDE is just a standard CIR model with an initial condition that depends on the value of the compound Hawkes after the jump has occurred. Since strong existence and pathwise uniquess is proven for the SDE that defines the CIR model, see \cite[Theorem 1.2.1]{AlfonsiAurélien2015ADaR}, one can properly interlace the solutions of the continuous path SDE and the jumps as they do in \cite[Chapter V.10, Theorem 57]{Pro1992} and prove strong existence and pathwise uniqueness for our SDE in \eqref{21}. In particular, this implies that $v$ is a positive process. 

\begin{prop}
Equation \eqref{21} has a pathwise unique strong solution. 
\end{prop}

Since the sizes of the jumps of the compound Hawkes process are strictly positive one expects that our variance is greater or equal than the Heston variance. We prove this property following the proof of \cite[Chapter 5.2.C, Proposition 2.18]{Karatzas1991}. The procedure is the same but in our case some extra computations will appear due to the jump contribution of the compound Hawkes process. Recall that the variance of the Heston model is strictly positive because the Feller condition is assumed to hold, see the reference from before \cite[Proposition 1.2.15]{AlfonsiAurélien2015ADaR}. In particular, we will prove that the process $v$ is also strictly positive.
\begin{prop}\label{p1}
Let $\widetilde{v}=\{\widetilde{v}_t, t\in[0,T]\}$ be the pathwise unique strong solution of 
\begin{align}\label{sde10}
    \widetilde{v}_t=v_0-\kappa\int_0^t\left(\widetilde{v}_s-\Bar{v}\right)ds+\sigma\int_0^t\sqrt{\widetilde{v}_s}dW_s.
\end{align}
Then,
\begin{align*}
    \PP\left(\left\{\omega\in\Omega : \widetilde{v}_t(\omega)\leq v_t(\omega) \ \forall t\in[0,T]\right\}\right)=1,
\end{align*}
where $v$ is the pathwise unique strong solution of \eqref{21}.
\end{prop}
\begin{proof}
See \cite[Theorem 1.2.1]{AlfonsiAurélien2015ADaR} for a reference about the solution of the SDE in \eqref{sde10}.
There exists a strictly decreasing sequence $\{a_n\}_{n=0}^\infty\subset(0,1]$ with $a_0=1$, $\lim_{n\rightarrow\infty}a_n=0$ and 
\begin{align*}
    \int_{a_n}^{a_{n-1}}\frac{du}{u}=n,
\end{align*}
for every $n\geq1$. Precisely, $a_n=\exp\left(-\frac{n(n+1)}{2}\right)$. 

For each $n\geq1$, there exists a continuous function $\rho_n$ on $\R$ with support in $(a_n,a_{n-1})$ so that
\begin{align}\label{ineq1}
    0\leq \rho_n(x)\leq\frac{2}{nx}
\end{align}
holds for every $x>0$ and $\int_{a_n}^{a_{n-1}}\rho_n(x)dx=1$.

Then, the function
\begin{align}\label{def}
    \psi_n(x)=\int_0^{|x|}\int_0^y\rho_n(u)dudy, 
\end{align} 
is even and twice continuously differentiable, with $|\psi_n'(x)|\leq1$ and $\lim_{n\rightarrow\infty}\psi_n(x)=|x|$ for $x\in\R$. Furthermore, $\{\psi_n\}_{n=1}^\infty$ is nondecreasing.

Next, define the nondecreasing function $\varphi_n(x)=\psi_n(x)\mathds{1}_{(0,\infty)}(x)$. Note that
    \begin{align*}
        \lim_{n\rightarrow\infty}\varphi_n(x)=\max\{x,0\}=:x^+.
    \end{align*}
Define $I_t:=\widetilde{v}_t-v_t$ for $t\in[0,T]$. Applying Itô formula we get
\begin{align*}
    \varphi_n(I_t) = \ &\varphi_n(0)+\int_0^t\varphi_n'(I_{s-})dI_s +\frac{1}{2}\int_0^t\varphi_n''(I_{s-})d[I]_s^{\text{c}} \\ 
    & +\sum_{0<s\leq t}\left[\varphi_n(I_s)-\varphi_n(I_{s-})-\varphi_n'(I_{s-})\Delta I_s\right].
\end{align*}
Note that 
\begin{align*}
    dI_s&=-\kappa(\widetilde{v}_s-v_s)ds+\sigma\left(\sqrt{\widetilde{v}_s}-\sqrt{v_s}\right)dW_s-\eta dL_s, \\
    d[I]_s&=\sigma^2\left(\sqrt{\widetilde{v}_s}-\sqrt{v_s}\right)^2ds+\eta^2d[L]_t.
\end{align*}
The latter implies that $d[I]_s^{\text{c}}=\sigma^2\left(\sqrt{\widetilde{v}_s}-\sqrt{v_s}\right)^2ds$ and $\Delta I_s=-\Delta v_s=-\eta \Delta L_s.$
Therefore, 
\begin{align*}
    \varphi_n(I_t)= \ &-\kappa\int_0^t\varphi_n'(I_s)I_sds+\sigma\int_0^t\varphi_n'(I_{s-})\left(\sqrt{\widetilde{v}_s}-\sqrt{v_s}\right)dW_s  \\
    &-\eta\int_0^t\varphi_n'(I_{s-})dL_s+\frac{\sigma^2}{2}\int_0^t\varphi_n''(I_s)\left(\sqrt{\widetilde{v}_s}-\sqrt{v_s}\right)^2ds \\
    &+\sum_{0<s\leq t}\left[\varphi_n(I_s)-\varphi_n(I_{s-})\right]+\eta\sum_{0<s\leq t}\varphi_n'(I_{s-})\Delta L_s.
\end{align*}
Note that $\eta\int_0^t\varphi_n'(I_{s-})dL_s=\eta\sum_{0<s\leq t}\varphi_n'(I_{s-})\Delta L_s.$
 Since $I_s-I_{s-}=-\eta\Delta L_s\leq0$ and $\varphi_n$ is a non-decreasing function, we have $\sum_{0<s\leq t}\left[\varphi_n(I_s)-\varphi_n(I_{s-})\right] \leq0.$
 
By \eqref{ineq1}, \eqref{def} and using that $|\sqrt{x}-\sqrt{y}|\leq\sqrt{|x-y|}$ for $x,y\geq0$ we obtain
\begin{align*}
    \varphi_n''(I_s)\left(\sqrt{\widetilde{v}_s}-\sqrt{v_s}\right)^2\leq \frac{2}{n}\frac{\left(\sqrt{\widetilde{v}_s}-\sqrt{v_s}\right)^2}{\widetilde{v}_s-v_s}\leq\frac{2}{n}.
\end{align*}
Due to the fact that $|\varphi_n'(x)|\leq 1$ we can conclude that
\begin{align}\label{comp10}
    \varphi_n(I_t)\leq \ &\kappa\int_0^tI_s^+ds+\sigma\int_0^t\varphi_n'(I_{s-})\left(\sqrt{\widetilde{v}_s}-\sqrt{v_s}\right)dW_s +\frac{\sigma^2 t}{n}.
\end{align}
In order to use the zero mean property of the Itô integral
\begin{align}\label{compC}
    \EE\left[\int_0^t\varphi_n'(I_{s-})\left(\sqrt{\widetilde{v}_s}-\sqrt{v_s}\right)dW_s\right]=0,
\end{align}
we need to check that
    \begin{align}\label{compB}
    \EE\left[\int_0^t\varphi_n'(I_s)^2\left(\sqrt{\widetilde{v}_s}-\sqrt{v_s}\right)^2ds\right]<\infty.
\end{align}
Now, 
\begin{align}\label{compA}
    \EE\left[\int_0^t\varphi_n'(I_s)^2\left(\sqrt{\widetilde{v}_s}-\sqrt{v_s}\right)^2ds\right]&= \EE\left[\int_0^t\varphi_n'(I_s)^2\left(\sqrt{\widetilde{v}_s}-\sqrt{v_s}\right)^2\mathds{1}_{\{I_s>0\}}ds\right] \notag\\
    & \leq \EE\left[\int_0^t\left(\sqrt{\widetilde{v}_s}-\sqrt{v_s}\right)^2\mathds{1}_{\{I_s>0\}}ds\right] \notag\\
    & \leq \EE\left[\int_0^t\widetilde{v}_sds\right] \notag\\
    & = \int_0^t\EE\left[\widetilde{v}_s\right]ds.
\end{align}
In \cite[Section 2, Equation (3)]{inverseheston} we see that 
\begin{align*}
    \widetilde{v}_t \sim \frac{e^{-\kappa t}}{k(t)}\chi_\delta^{'2}\left(k(t)\right) \ \ \text{with} \ \ k(t)=\frac{4\kappa e^{-\kappa t}}{\sigma^2(1-e^{-\kappa t})}v_0 \ \ \text{and} \ \ \delta=\frac{4\kappa\Bar{v}}{\sigma^2},
\end{align*}
where $\chi_\delta^{'2}(k(t))$ denotes a noncentral chi-square random variable with $\delta$ degrees of freedom and noncentrality parameter $k(t)$. Then, 
\begin{align}\label{exp}
\EE\left[\widetilde{v}_t\right]=\frac{\sigma^2(1-e^{-\kappa t})}{4\kappa}\left(\delta+k(t)\right)=v_0e^{-\kappa t}+\Bar{v}(1-e^{-\kappa t}).
\end{align}
Since $t\mapsto\EE\left[\widetilde{v}_t\right]$ is continuous function on $[0,T]$, the integral in \eqref{compA} is finite, which  implies that the integral in \eqref{compB} is finite and \eqref{compC} holds. Taking expectations in \eqref{comp10} we have
\begin{align*}
    \EE\left[\varphi_n(I_t)\right]\leq\kappa\int_0^t\EE\left[I_s^+\right]+\frac{\sigma^2 t}{n}.
\end{align*}
Then, sending $n$ to infinity yields $\EE\left[I_t^+\right]\leq\kappa\int_0^t\EE\left[I_s^+\right]ds.$ One can check that
\begin{align}\label{gron}
    \int_0^t\left|\EE\left[I_s^+\right]\right|ds<\infty.
\end{align}
Actually
$$\int_0^t\left|\EE\left[I_s^+\right]\right|ds = \int_0^t\EE\left[I_s^+\right]ds = \int_0^t\EE\left[I_s\mathds{1}_{\{I_s>0\}}\right]ds \leq \int_0^t\EE\left[\widetilde{v}_s\right]ds<\infty,$$
where the last integral is finite as we have seen before.

We can apply a version of Gronwall's inequality where only condition \eqref{gron} is required and we get $\EE[I_t^+]=0$ for all $t\in[0,T]$. This means that
\begin{align*}
    \PP\left(\{\omega\in\Omega: \widetilde{v}_t(\omega)\leq v_t(\omega)\}\right)=1 \hspace{1cm} \forall t\in[0,T].
\end{align*}
Since the sample paths of $\widetilde{v}$ are continuous and the sample paths of $v$ are càdlàg we get that
\begin{align*}
    \PP\left(\left\{\omega\in\Omega : \widetilde{v}_t(\omega)\leq v_t(\omega) \ \forall t\in[0,T]\right\}\right)=1.
\end{align*}
\end{proof}
\begin{cor}\label{pos}
The variance $v=\{v_t, t\in[0,T]\}$ is a strictly positive process. 
\end{cor}
\begin{proof}
Recall that we have assumed that the Feller condition $2\kappa \Bar{v}\geq\sigma^2$ is satisfied. Then, the process $\widetilde{v}$ defined by \eqref{sde10} is strictly positive, see \cite[Proposition 1.2.15]{AlfonsiAurélien2015ADaR}. Finally, by Proposition \ref{p1} we have that
\begin{align*}
    \PP\left(\left\{\omega\in\Omega : \widetilde{v}_t(\omega)\leq v_t(\omega) \ \forall t\in[0,T]\right\}\right)=1.
\end{align*}
We conclude that the process $v$ is also strictly positive.
\end{proof}
\section{Risk neutral probability measures}\label{sec:change}
To prove the existence of a family of risk neutral probability measures we follow the classical approach where we employ Girsanov's theorem in connection with the Novikov's condition. Thus, we need to study the integrability of the exponential of the integrated variance, that is, what values $c>0$, if any, satisfy
\begin{align*}
    \EE\left[\exp\left(c\int_0^Tv_udu\right)\right]<\infty.
\end{align*}
By Corollary \ref{pos}, $v$ is strictly positive and the previous expectation is finite for $c\leq0$ but for our applications is essential that $c$ can be strictly positive.
Looking at the proof of \cite[Lemma 3.1]{onchanges} one can see that for $c\leq \frac{\kappa^2}{2\sigma^2}$ the following holds
\begin{align}\label{5}
    \EE\left[\exp\left(c\int_0^T\widetilde{v}_udu\right)\right]<\infty,
\end{align}
where $\widetilde{v}$ is the standard Heston volatility given by
\begin{align*}
    \widetilde{v}_t=v_0-\kappa\int_0^t\left(\widetilde{v}_s-\Bar{v}\right)ds+\sigma\int_0^t\sqrt{\widetilde{v}_s}dW_s.
\end{align*}
The procedure to prove that \eqref{5} holds is to show that
\begin{align}\label{6}
\EE\left[\exp\left(c\int_0^T\widetilde{v}_udu\right)\right]\leq\exp\left(-(\kappa \Bar{v})\Phi(0)-v_0\psi(0)\right)<\infty,
\end{align}
where $\Phi$ and $\psi$ satisfy the following Riccati ODE 
\begin{align}\label{psi}
    \psi'(t) &=\frac{\sigma^2}{2}\psi^2(t)+\kappa\psi(t)+c \\
    -\Phi'(t) & =\psi(t) \notag\\
    \psi(T) &=\Phi(T)=0. \notag
\end{align}
Due to the jump contribution of the compound Hawkes process, this procedure is more delicate in our model. However we can obtain a similar bound as in \eqref{6} with an additional function that will be the solution of an ODE where the moment generating function of $J_1$ and the Hawkes parameters $\alpha$ and $\beta$ are involved. 

We start making an assumption on the moment generating function of $J_1$. We write $M_{J}$ for the moment generating function of $J_1$, that is, $M_{J}(t)=\EE\left[\exp\left(tJ_1\right)\right]$. Since $J_1$ is strictly positive, $M_J$ is well-defined at least on the interval $(-\infty,0]$. 
\begin{ass}\label{as} 
There exists $\epsilon_J>0$ such that $M_J$ is well defined on $(-\infty,\epsilon_J)$ and it is the maximal domain in the sense that
\begin{align*}
    \lim_{t\rightarrow\epsilon_J^-}M_J(t)=\infty.
\end{align*}
Since $\epsilon_J>0$, all positive moments of $J_1$ are finite. 
\end{ass}

Note that this is the case of the gamma distribution, chi-squared distribution, uniform distribution and others. 

We start studying the functions that will appear in a bound like \eqref{6} for our variance. To find the ODEs that define those functions, one can consider the process $M=\{M(t), t\in[0,T]\}$ defined by
\begin{align*}
    M(t)=\exp\left(F(t)+G(t)v_t+H(t)\lambda_t+c\int_0^tv_udu\right),
\end{align*}
for some unknown functions $F,G,H\colon[0,T]\to\R$ satisfying $F(T)=G(T)=H(T)=0$. Note that 
\begin{align*}
    M(T)=\exp\left(c\int_0^Tv_udu\right).
\end{align*}
Hence, $\EE[M(T)]$ is exactly the expectation that we want to study. Now, if we can find $F,G$ and $H$ such that $M$ is a local martingale, since it is non-negative, it will be a supermartingale and then
\begin{align*}
    \EE\left[\exp\left(c\int_0^Tv_udu\right)\right]=\EE[M(T)]\leq M(0)=\exp\left(F(0)+G(0)v_0+H(0)\lambda_0\right),
\end{align*}
where the last expression will be finite as long as $F,G$ and $H$ exist and are well defined on $[0,T]$.

To make $M$ a local martingale one can apply Itô's formula to the process $M$ and equate all the drift terms to $0$. By this procedure one will obtain the ODEs for $F,G$ and $H$. This procedure is formally done in Proposition \ref{novikov}. In the next lemma we first study the existence of solutions of the ODEs that will appear later. Note that the ODE for G is slightly different from the ODE of $\psi$ in \eqref{psi}, this depends on whether one considers the expectation in \eqref{5} with the parameter $c$ or $-c$ and on how one defines the process $M$. The proof of the next lemma is deferred to the appendix because is rather technical and based on ODE theory.  
\begin{lem}\label{lode2} For $c\leq\frac{\kappa^2}{2\sigma^2}$, define $D(c):=\sqrt{\kappa^2-2\sigma^2c}$, $\Lambda(c):=\frac{2\eta c\left(e^{D(c)T}-1\right)}{D(c)-\kappa+\left(D(c)+\kappa\right)e^{D(c)T}}$
and
\begin{align*}
    c_l:=\sup\left\{c\leq\frac{\kappa^2}{2\sigma^2}: \Lambda(c)< \epsilon_J \hspace{0.3cm}\text{and}\hspace{0.3cm} M_J\left(\Lambda(c)\right)\leq\frac{\beta}{\alpha}\exp\left(\frac{\alpha}{\beta}-1\right) \right\}.
\end{align*}
Then, $0<c_l\leq\frac{\kappa^2}{2\sigma^2}$ and for $c< c_l$, 
\begin{enumerate}[(i)]
    \item The ODE
\begin{align}\label{ode1l}
    G'(t) & =-\frac{1}{2}\sigma^2G^2(t)+\kappa G(t)-c  \\
    G(T) & = 0 \notag
\end{align}
has a unique solution in the interval $[0,T]$. The solution is strictly decreasing and it is given by 
\begin{align*}
    G(t)=\frac{2c\left(e^{D(c)(T-t)}-1\right)}{D(c)-\kappa+\left(D(c)+\kappa\right)e^{D(c)(T-t)}}.
\end{align*}
    \item The function $t\mapsto M_J(\eta G(t))$ is well defined for $t\in[0,T]$.
    \item Define $U:=\sup_{t\in [0,T]}M_J(\eta G(t))$. Then, $U=M_J(\eta G(0))$ and 
\begin{align*}
    1<U\leq \frac{\beta}{\alpha}\exp\left(\frac{\alpha}{\beta}-1\right).
\end{align*} 
\item The ODE
\begin{align}\label{ode2l}
    H'(t) & =\beta H(t)-M_J\left(\eta G(t)\right)\exp\left(\alpha H(t)\right) +1 \\
    H(T) & = 0 \notag
\end{align}
has a unique solution in $[0,T]$.
\end{enumerate} 
\end{lem}
\begin{proof}
See Lemma \ref{Alode2} in the appendix. 
\end{proof}

\begin{obs}
One could also assume that the domain of $M_J$ is big enough to make the function $t\mapsto M_J(\eta G(t))$ well defined on the interval $[0,T]$ for any value of $c\leq\frac{\kappa^2}{2\sigma^2}$.  Then one can prove the existence of $H$ by applying the Picard-Lindelöf theorem (see \cite[Chapter II, Theorem 1.1]{hartman}) which yields the following conditions
\begin{align*}
    c\leq\frac{\kappa^2}{2\sigma^2} \ \ \text{and} \ \ \beta+f(U)\alpha\leq \frac{1}{T},
\end{align*}
where $f(U)$ is some function of $U$. Therefore, there would be more admissible values of $c$ but $\alpha$ and $\beta$ would have to satisfy an inequality involving $T$, which is quite restrictive and for the sake of generality we prefer to avoid. 
\end{obs}

Note that, a priori, there is not an explicit expression of $c_l$ in the previous lemma because it is not possible to solve the inequalities 
\begin{align*}
    \Lambda(c)< \epsilon_J \hspace{0.5cm}\text{  and  }\hspace{0.5cm} M_J\left(\Lambda(c)\right)\leq\frac{\beta}{\alpha}\exp\left(\frac{\alpha}{\beta}-1\right).
\end{align*}
Moreover, $c_l$ depends on $M_J$, $\eta$, $\kappa$, $\sigma$, $T$, $\epsilon_J$, $\alpha$ and $\beta$. However, one can get a suboptimal and explicit value of $c_l$ that does not depend on $T$. We obtain an expression for that suboptimal value and we give some examples.
\begin{cor} Define $c_s$ by 
\label{explicit}
\begin{align*}
    c_s=\min\left\{\frac{\kappa\epsilon_J}{2\eta},\frac{\kappa }{2\eta}M_J^{-1}\left(\frac{\beta}{\alpha}\exp\left(\frac{\alpha}{\beta}-1\right)\right),\frac{\kappa^2}{2\sigma^2}\right\}.
\end{align*}
Then, $0<c_s<c_l$.
\end{cor}
\begin{proof}
See Lemma \ref{Aexplicit} in the appendix. 
\end{proof}

Note that for any $c<c_s$ we can apply Lemma \ref{lode2} because $c_s<c_l$. We now give some examples of the value $c_s$. 
\begin{exam}
Some examples:
\begin{enumerate}[(i)]
    \item If $J_1\sim\text{Exponential}(\lambda)$, then 
\begin{align*}
    c_s=\min\left\{\frac{\kappa\lambda}{2\eta}\left(1-\frac{\alpha}{\beta}\exp\left(1-\frac{\alpha}{\beta}\right)\right),\frac{\kappa^2}{2\sigma^2}\right\}.
\end{align*}
\item If $J_1\sim\text{Gamma}(\mu,\lambda)$ with $\mu,\lambda>0$ as the shape and the rate, respectively. Then
\begin{align*}
    c_s=\min\left\{\frac{\kappa\lambda }{2\eta}\left(1-\frac{1}{\left(\frac{\beta}{\alpha}\exp\left(\frac{\alpha}{\beta}-1\right)\right)^{1/\mu}}\right),\frac{\kappa^2}{2\sigma^2}\right\}.
\end{align*}
\item If $J_1=j>0$, then
\begin{align*}
    c_s=\min\left\{\frac{\kappa}{2\eta j}\left(\ln\left(\frac{\beta}{\alpha}\right)+\frac{\alpha}{\beta}-1\right),\frac{\kappa^2}{2\sigma^2}\right\}.
\end{align*}
\end{enumerate}
\end{exam}
\begin{proof}
See Example \ref{exA} in the appendix. 
\end{proof}

We have studied under which conditions the previous ODEs have a well defined solution on the interval $[0,T]$. We proceed studying the integrability of the exponential of the integrated variance following the method commented in the beginning of this section. 
\begin{prop}\label{novikov}
Let $c<c_l$, then
\begin{align*}
    \EE\left[\exp\left(c\int_0^Tv_udu\right)\right]<\infty.
\end{align*}
\end{prop}
\begin{proof}
By Corollary \ref{pos}, $v$ is strictly positive and the expectation is finite for $c\leq0$. We focus on the case when $0<c<c_l$. We first define the function $f\colon[0,T]\times\R^3\to\R$ 
\begin{align*}
    f(t,x,y,z)=\exp\left(F(t)+G(t)x+H(t)y+cz\right),
\end{align*}
where $G$ and $H$ are the solution of the differential equations given in Lemma \ref{lode2}, that is, 
\begin{align}
    G'(t) & =-\frac{1}{2}\sigma^2G^2(t)+\kappa G(t)-c \label{node1l} \\
    G(T) & = 0 \notag
\end{align}
and 
\begin{align}
    H'(t) & =\beta H(t)-M_J\left(\eta G(t)\right)\exp\left(\alpha H(t)\right) +1 \label{node2l}\\
    H(T) & = 0, \notag
\end{align}
and $F$ is given by 
\begin{align}\label{ode1}
    F'(t) & =-\kappa\Bar{v}G(t)-\beta\lambda_0H(t) \\
    F(T) & =0\notag.
\end{align}
Note that with the assumption we have made on $c$, the functions $F$, $G$ and $H$ are well defined on $[0,T]$.

We also define the integrated variance $V_t:=\int_0^tv_udu$, $Y_t:=(t,v_t,\lambda_t,V_t)$  and the process $M=\{M(t), t\in[0,T]\}$ by $M(t)=f(t,v_t,\lambda_t,V_t)=f(Y_t)$. Applying Itô formula to the process $M$ we get
\begin{align*}
    M(t)= \ &M(0)+\int_0^t\partial_tf(Y_{s-})ds+\int_0^t\partial_xf(Y_{s-})dv_s+\int_0^t\partial_y f(Y_{s-})d\lambda_s \\
    &+\int_0^t\partial_zf(Y_{s-})dV_s+\frac{1}{2}\int_0^t\partial^2_{xx}f(Y_{s-})d[v]_s^{\text{c}} \\
    &+\sum_{0<s\leq t}\left[f(Y_s)-f(Y_{s-})-\partial_xf(Y_{s-})\Delta v_s-\partial_y f(Y_{s-})\Delta\lambda_s\right].
\end{align*}
We have used that
\begin{align*}
    [\lambda]_t & =\alpha^2[N]_t=\alpha^2N_t \implies [\lambda]_t^{\text{c}}  =0, \\
    [V]_t & = 0, \\
    [v,\lambda]_t & =\alpha\eta[L,N]_t=\alpha\eta L_t\implies [v,\lambda]_t^{\text{c}}  =0, \\
    [v,V]_t & =[\lambda,V]_t =0.
\end{align*}
Moreover, 
\begin{align*}
    [v]_t & =\int_0^t\sigma^2v_sds+\eta^2[L]_t \implies [v]_t^{\text{c}}=\int_0^t\sigma^2v_sds, \\
    dv_t & =-\kappa(v_t-\Bar{v})dt+\sigma\sqrt{v_t}dW_t+\eta dL_t, \\
    d\lambda_t & =-\beta(\lambda_t-\lambda_0)dt+\alpha dN_t.
\end{align*}

Hence, 
\begin{align*}
    M(t)-M(0) = &\int_0^t\Big[\partial_tf(Y_{s})-\kappa\partial_xf(Y_{s})(v_s-\Bar{v})-\beta\partial_y f(Y_{s})(\lambda_s-\lambda_0) \\
    &+\partial_zf(Y_{s})v_s+\frac{1}{2}\partial_{xx}^2f(Y_{s})\sigma^2v_s\Big]ds\\ 
        &+\int_0^t\partial_xf(Y_{s-})\sigma\sqrt{v_s}dW_s+\int_0^t\partial_yf(Y_{s-})\eta dL_s+\int_0^t\partial_y f(Y_{s-})\alpha dN_s \\
    &+\sum_{0<s\leq t}\left[f(Y_s)-f(Y_{s-})-\partial_xf(Y_{s-})\Delta v_s-\partial_y f(Y_{s-})\Delta \lambda_s\right].
\end{align*}
Since $\Delta v_s=\eta\Delta L_s$ and $\Delta \lambda_s=\alpha\Delta N_s$, then
\begin{align*}
    \int_0^t\partial_xf(Y_{s-})\eta dL_s  & = \sum_{0<s\leq t}\partial_xf(Y_{s-})\Delta v_s \\
    \int_0^t\partial_y f(Y_{s-})\alpha dN_s & = \sum_{0<s\leq t}\partial_y f(Y_{s-})\Delta \lambda_s,
\end{align*}
and we get
\begin{align*}
    M(t)-M(0) = &\int_0^t\Big[\partial_tf(Y_{s})-\kappa\partial_xf(Y_{s})(v_s-\Bar{v})-\beta\partial_y f(Y_{s})(\lambda_s-\lambda_0) \\
    &+\partial_zf(Y_{s})v_s+\frac{1}{2}\partial_{xx}^2f(Y_{s})\sigma^2v_s\Big]ds\\ 
        &+\int_0^t\partial_xf(Y_{s-})\sigma\sqrt{v_s}dW_s +\sum_{0<s\leq t}\left[f(Y_s)-f(Y_{s-})\right].
\end{align*}
Next, we can write
\begin{align*}
    \sum_{0<s\leq t}\left[f(Y_s)-f(Y_{s-})\right] & =\sum_{0<s\leq t}\left[f(s,v_{s-}+\Delta v_s, \lambda_{s-}+\Delta \lambda_s,V_s)-f(Y_{s-})\right] \\
    & = \sum_{0<s\leq t}\left[f(s,v_{s-}+\eta \Delta L_s, \lambda_{s-}+\alpha\Delta N_s,V_s)-f(Y_{s-})\right] \\
    & = \sum_{0<s\leq t} g(s,\Delta L_s,\Delta N_s),
\end{align*}
where 
\begin{align*}
    g(s,u_1,u_2):=f(s,v_{s-}+\eta u_1, \lambda_{s-}+\alpha u_2,V_s)-f(Y_{s-}).
\end{align*} 
We now define $U_s=(L_s,N_s)$ and for $t\in[0,T]$, $A\in\mathcal{B} (\R^2\setminus\{0,0\})$ 
\begin{align*}
    N^U(t,A)=\#\{0<s\leq t, \Delta U_s\in A\}.
\end{align*}
We add and subtract the compensator of the counting measure $N^U$ to split the expression into a local martingale plus a predictable process of finite variation. Note that the compensator of the Hakwes process is given by $\Lambda^N_t=\int_0^t\lambda_udu$, see \cite[Theorem 3]{Hawkesfinance}. One can check the the compensator of the compound Hawkes process is given by $\Lambda^L_t=\EE[J_1]\int_0^t\lambda_udu$. Thus,
\begin{align*}
    \sum_{0<s\leq t}\left[f(Y_s)-f(Y_{s-})\right] & = \int_0^t\int_{(0,\infty)^2}g(s,u_1,u_2)N^U(ds,du) \\
    & = \int_0^t\int_{(0,\infty)^2} g(s,u_1,u_2)\left(N^U(ds,du)-\lambda_sP_{J_1}(du_1)\delta_1(du_2)ds\right) \\ &+\int_0^t\int_{(0,\infty)^2} g(s,u_1,u_2)\lambda_sP_{J_1}(du_1)\delta_1(du_2)ds.
\end{align*}
Note that
\begin{align*}
    \int_0^t\int_{(0,\infty)^2} g(s,u_1,u_2)\lambda_sP_{J_1}(du_1)\delta_1(du_2)ds = \int_0^t\int_{(0,\infty)} g(s,u_1,1)\lambda_sP_{J_1}(du_1)ds.
\end{align*}
We  conclude that 
    \begin{align*}
    M(t)-M(0) = &\int_0^t\Big[\partial_tf(Y_{s})-\kappa\partial_xf(Y_{s})(v_s-\Bar{v})-\beta\partial_y f(Y_{s})(\lambda_s-\lambda_0) \\
    &+\partial_zf(Y_{s})v_s+\frac{1}{2}\partial_{xx}^2f(Y_{s})\sigma^2v_s+\int_{(0,\infty)} g(s,u_1,1)\lambda_sP_{J_1}(du_1)\Big]ds \\
    &+\int_0^t\partial_xf(Y_{s-})\sigma\sqrt{v_s}dW_s \\
        &+\int_0^t\int_{(0,\infty)} g(s,u_1,1)\left(N^U(ds,du)-\lambda_sP_{J_1}(du_1)ds\right).
\end{align*}
Recall that $f(t,x,y,z)=\exp\left(F(t)+G(t)x+H(t)y+cz\right)$, thus,
\begin{align*}
    \partial_tf(t,x,y,z) & =\left(F'(t)+G'(t)x+H'(t)y\right)f(t,x,y,z), \\
    \partial_xf(t,x,y,z) & =G(t)f(t,x,y,z), \\ 
     \partial_y f(t,x,y,z) & =H(t)f(t,x,y,z), \\
      \partial_zf(t,x,y,z)& =cf(t,x,y,z).
\end{align*}
Furthermore,
\begin{align*}
   g(s,u_1,1) & =f(s,v_{s-}+\eta u_1,\lambda_{s-}+\alpha,V_s)-f(s,v_{s-},\lambda_{s-},V_s) \\
   & = f(s,v_{s-},\lambda_{s-},V_s)\left[\exp\left(\eta u_1G(s)\right)\exp\left(\alpha H(s)\right)-1\right].
\end{align*}
Therefore,
\begin{align*}
    \int_{(0,\infty)} g(s,u_1,1)\lambda_sP_{J_1}(du_1) = f(Y_{s-})\lambda_s\left[M_J\left(\eta G(s)\right)\exp\left(\alpha H(s)\right)-1\right].
\end{align*}
Using the specific form of the derivatives and the previous result we see that the drift part of $M(t)-M(0)$ vanishes. That is:
\begin{itemize}
\item Coefficient multiplying $v$ in the drift:
\begin{align*}
    f(Y_{t})\left[G'(t)-\kappa G(t)+\frac{1}{2}G(t)^2\sigma^2+c\right]=0,
\end{align*}
where we have replaced equation \eqref{node1l}. 
\item Coefficient multiplying $\lambda$ in the drift:
\begin{align*}
    f(Y_{t})\left[H'(t)-\beta H(t)+M_J\left(\eta G(t)\right)\exp\left( \alpha H(t)\right)-1\right] =0 ,
\end{align*}
where we have replaced equation \eqref{node2l}.
\item Free coefficient in the drift:
\begin{align*}
    f(Y_{t})\left[F'(t)+\kappa\Bar{v}G(t)+\beta\lambda_0H(t)\right]=0,
\end{align*}
where we have replaced equation \eqref{ode1}.
\end{itemize}
Therefore, the process $M$ can be written as
\begin{align*}
    M(t)  = & \ M(0) +\int_0^t\partial_vf(Y_{s-})\sigma\sqrt{v_s}dW_s \\
        &+\int_0^t\int_{(0,\infty)} g(s,u_1,1)\left(N^U(ds,du)-\lambda_sP_{J_1}(du_1)ds\right).
\end{align*}
We conclude that $M$ is a local martingale. Moreover, since $M$ is non-negative, it is a supermartingale. Then,
\begin{align*}
    \EE[M(T)] &=\EE\left[\exp\left(F(T)+G(T)v_T+H(T)\lambda_T+cV_T\right)\right] \\
    & = \EE\left[\exp\left(c\int_0^Tv_udu\right)\right] \\ 
    & \leq M(0) =\exp\left(F(0)+G(0)v_0+H(0)\lambda_0\right) <\infty.
\end{align*}
Note that $F(0),G(0),H(0)<\infty$ because the functions $F,G$ and $H$ are well defined on the entire interval $[0,T]$.
\end{proof}

We write $\widehat{S}=\{\widehat{S}_t, t\in[0,T]\}$ for the discounted stock price, that is, $\widehat{S}_t=e^{-rt}S_t$. We prove the existence of a family of equivalent local martingale measures $\QQ(a)$ parametrized by a number $a$. Namely, $\QQ(a)\sim\PP$ is such that $\widehat{S}$ is a $(\mathcal{F},\QQ(a))$-local martingale.

\begin{thm}\label{risk} Let $a\in\RR$ and define $\theta_t^{(a)}:=\frac{1}{\sqrt{1-\rho^2}}\left(\frac{\mu_t-r}{\sqrt{v_t}}-a\rho\sqrt{v_t}\right)$, 
\begin{align*}
    Y_t^{(a)} &:=\exp\left(-\int_0^t\theta_s^{(a)}dB_s-\frac{1}{2}\int_0^t(\theta_s^{(a)})^2ds\right), \\
    Z_t^{(a)} &:=\exp\left(-a\int_0^t\sqrt{v_s}dW_s-\frac{1}{2}a^2\int_0^tv_sds\right)
\end{align*}
and $X_t^{(a)}:=Y_t^{(a)}Z_t^{(a)}$. The set 
\begin{align}\label{elmmset}
    \mathcal{E}:=\left\{\QQ(a) \hspace{0.2cm}\text{given by}\hspace{0.2cm} \frac{d\QQ(a)}{d\PP}=X_T^{(a)} \hspace{0.2cm}\text{with}\hspace{0.2cm} |a|<\sqrt{2c_l}\right\} 
\end{align}
is a set of equivalent local martingale measures. 
\end{thm}
\begin{proof}
Define the process $(B^{\QQ(a)},W^{\QQ(a)})=\{(B_t^{\QQ(a)},W_t^{\QQ(a)}), t\in[0,T]\}$, by
    \begin{align}\label{BMQ}
    dB_t^{\QQ(a)}&=dB_t+\theta_t^{(a)}dt, \notag\\
    dW_t^{\QQ(a)}&=dW_t+a\sqrt{v_t}dt.
\end{align}
The dynamics of the stock is now given by
\begin{align*}
    \frac{dS_t}{S_t}=\left[\mu_t-\sqrt{v_t}\left(\sqrt{1-\rho^2} \theta_t^{(a)}+a\rho \sqrt{v_t}\right)\right]dt+\sqrt{v_t}\left(\sqrt{1-\rho^2} dB_t^{\QQ(a)}+\rho dW_t^{\QQ(a)}\right).
\end{align*}
Note that 
\begin{align}\label{cond}
    \mu_t-\sqrt{v_t}\left(\sqrt{1-\rho^2} \theta_t^{(a)}+a\rho \sqrt{v_t}\right)=r,
\end{align}
which is a necessary condition for $\QQ(a)$ to be an equivalent local martingale measure. The choice of the market price of risk processes $\theta^{(a)}$ and $a\sqrt{v_t}$ is the same as in \cite[Equation (3.4) and (3.7)]{onchanges}. One of the reasons to make that choice under the standard Heston model is to maintain the same variance dynamics after the change of measure.

To apply Girsanov's theorem we need to check that the process $X^{(a)}$ is a $(\mathcal{F},\PP)$-martingale.
Since $X^{(a)}$ is a positive $(\mathcal{F},\PP)$-local martingale with $X^{(a)}_0=1$, it is a $(\mathcal{F},\PP)$-supermartingale and it is a $(\mathcal{F},\PP)$-martingale if and only if
\begin{align}\label{cond2}
    \EE\left[X^{(a)}_T\right]=1.
\end{align}
Using that $Z^{(a)}_T$ is $\mathcal{F}_T^{W}\vee\mathcal{F}_T^L$-measurable we have
\begin{align}\label{eq3}
    \EE\left[X^{(a)}_T\right]=\EE\left[Y^{(a)}_TZ^{(a)}_T\right]=\EE\left[\EE\left[Y^{(a)}_TZ^{(a)}_T|\mathcal{F}_T^{W}\vee\mathcal{F}_T^L\right]\right]=\EE\left[Z^{(a)}_T\EE\left[Y^{(a)}_T|\mathcal{F}_T^{W}\vee\mathcal{F}_T^L\right]\right].
\end{align}
By Corollary \ref{pos} the variance process $v$ is strictly positive. This implies that $\int_0^T(\theta_s^{(a)})^2ds<\infty,  \PP$-a.s, and since  $\theta^{(a)}$ is $\{\mathcal{F}_t^{W}\vee\mathcal{F}_t^L\}_{t\in[0,T]}$-adapted,
\begin{align*}
    Y^{(a)}_T|\mathcal{F}_T^{W}\vee\mathcal{F}_T^L\sim\text{Lognormal}\left(-\frac{1}{2}\int_0^T(\theta_s^{(a)})^2ds,\int_0^T(\theta_s^{(a)})^2ds\right),
\end{align*}
and we obtain that 
    $\EE[Y^{(a)}_T|\mathcal{F}_T^{W}\vee\mathcal{F}_T^L]=1.$
Therefore, replacing the last expression in \eqref{eq3} we obtain $\EE[X^{(a)}_T]=\EE[Z^{(a)}_T]$ and we only need to check that $\EE[Z^{(a)}_T]=1$. Since $|a|<\sqrt{2c_l}$, $\frac{1}{2}a^2< c_l$, we can use Proposition \ref{novikov} and get that
\begin{align}\label{nov}
    \EE\left[\exp\left(\frac{1}{2}a^2\int_0^Tv_udu\right)\right]<\infty.
\end{align}
Hence, Novikov's condition is satisfied, $Z^{(a)}$ is a $(\mathcal{F},\PP)$-martingale and we conclude that
$\EE[X^{(a)}_T]=\EE[Z^{(a)}_T]=1.$

Then, $X^{(a)}$ is a $(\mathcal{F},\PP)$-martingale, $\QQ(a)\sim\PP$ is an equivalent probability measure defined by
\begin{align*}
    \frac{d\QQ}{d\PP}&=X^{(a)}_T, \\ dX^{(a)}_t&=X^{(a)}_t\left[-\theta_t^{(a)}dB_t-a\sqrt{v_t}dW_t\right]
\end{align*}
and $(B^{\QQ(a)},W^{\QQ(a)})$ defined in \eqref{BMQ} is a two-dimensional standard $(\mathcal{F},\QQ(a))$-Brownian motion. The dynamics of the stock under $\QQ(a)$ is given by
\begin{align*}
    \frac{dS_t}{S_t}=rdt+\sqrt{v_t}\left(\sqrt{1-\rho^2} dB_t^{\QQ(a)}+\rho dW_t^{\QQ(a)}\right).
\end{align*}
This implies that the discounted stock $\widehat{S}$ is a $(\mathcal{F},\QQ(a))$-local martingale. Therefore, the set $\mathcal{E}$ defined in \eqref{elmmset} is a set of equivalent local martingale measures. 
\end{proof}
\begin{obs}
As we said in the beginning of this section, since $\frac{1}{2}a^2$ is multiplying the integrated variance in \eqref{nov} it is crucial that $c_l$ is strictly positive in Lemma \ref{lode2}.
\end{obs}
\begin{obs}
The dynamics of the variance under $\QQ(a)\in\mathcal{E}$ is given by 
\begin{align}\label{underQ}
    dv_t & =-\kappa\left(v_t-\Bar{v}\right)dt+\sigma\sqrt{v_t}\left(dW_t^{\QQ(a)}-a\sqrt{v_t}dt\right)+\eta dL_t \notag\\
    & =-\left(\kappa\left(v_t-\Bar{v}\right)+a\sigma v_t\right)dt+\sigma\sqrt{v_t}dW_t^{\QQ(a)}+\eta dL_t  \notag\\
    & =-\kappa^{(a)}\left(v_t-\Bar{v}^{(a)}\right)dt+\sigma\sqrt{v_t}dW_t^{\QQ(a)}+\eta dL_t,
\end{align}
where $\kappa^{(a)}=\kappa+a\sigma$ and $\Bar{v}^{(a)}=\frac{k\Bar{v}}{k+a\sigma}$. 
\end{obs}

So far, we have proven that there exists a set of equivalent local martingales measures. However, we need to study when those measures are actually equivalent martingale measures. We prove that under the condition $\rho^2<c_l$, there exists a subset of $\mathcal{E}$ of equivalent martingale measures. The condition $\rho^2<c_l$ may look quite restrictive. However, in \cite[Theorem 3.6]{onchanges} an inequality involving the correlation factor $\rho$ also appears in the proof of the existence of equivalent martingales measures in the standard Heston model.

\begin{thm}
If $\rho^2<c_l$, the set 
\begin{align}\label{choicea}
\mathcal{E}_m:=\left\{\QQ(a)\in\mathcal{E}: |a|<\min\left\{\frac{\sqrt{2c_l}}{2},\sqrt{c_l-\rho^2}\right\}\right\}
\end{align}
is a set of equivalent martingale measures.
\end{thm}
\begin{proof}
Let $\QQ(a)\in\mathcal{E}_m\subset\mathcal{E}$, by Theorem \ref{risk} $\widehat{S}$ is a $(\mathcal{F},\QQ(a))$-local martingale with the following dynamics 
\begin{align*}
    \frac{d\widehat{S}_t}{\widehat{S}_t}=\sqrt{v_t}\left(\sqrt{1-\rho^2}dB_t^{\QQ(a)}+\rho dW_t^{\QQ(a)}\right).
\end{align*}
Hence, 
\begin{align*}
    \widehat{S}_t=S_0\exp\left(\sqrt{1-\rho^2}\int_0^t\sqrt{v_s}dB_s^{\QQ(a)}+\rho\int_0^t\sqrt{v_s}dW_s^{\QQ(a)}-\frac{1}{2}\int_0^tv_sds\right).
\end{align*}
Since $\widehat{S}$ is a positive $\left(\mathcal{F},\QQ(a)\right)$-local martingale with $\widehat{S}_0=S_0$, it is a $\left(\mathcal{F},\QQ(a)\right)$-supermartingale and it is a $\left(\mathcal{F},\QQ(a)\right)$-martingale if and only if 
$\EE^{\QQ(a)}[\widehat{S}_T]=S_0.$ 

Similarly as we did in Theorem \ref{risk}, we define the processes
\begin{align*}
    Y_t^{\QQ(a)}:&=\exp\left(\sqrt{1-\rho^2}\int_0^t\sqrt{v_s}dB_s^{\QQ(a)}-\frac{1-\rho^2}{2}\int_0^tv_sds\right) \\
    Z_t^{\QQ(a)}:&=\exp\left(\rho\int_0^t\sqrt{v_s}dW_s^{\QQ(a)}-\frac{\rho^2}{2}\int_0^tv_sds\right).
    \end{align*}
Thus, $\widehat{S}_t=S_0Y_t^{\QQ(a)} Z_t^{\QQ(a)}$. Using that $Z_T^{\QQ(a)}$ is $\mathcal{F}_T^{W^{\QQ(a)}}\vee\mathcal{F}_T^L$-measurable we have
\begin{align}\label{cond.1}
    \EE^{\QQ(a)}[\widehat{S}_T] & =S_0\EE^{\QQ(a)}\left[Y_T^{\QQ(a)} Z_T^{\QQ(a)}\right]=S_0\EE^ {\QQ(a)}\left[\EE^{\QQ(a)}\left[Y_T^{\QQ(a)} Z_T^{\QQ(a)}|\mathcal{F}_T^{W^{\QQ(a)}}\vee\mathcal{F}_T^L\right]\right] \notag \\
    & = S_0\EE^{\QQ(a)}\left[Z_T^{\QQ(a)}\EE^{\QQ(a)}\left[Y_T^{\QQ(a)}|\mathcal{F}_T^{W^{\QQ(a)}}\vee\mathcal{F}_T^L\right]\right].
\end{align}
Since $\int_0^Tv_udu<\infty, {\QQ(a)}$-a.s. and $v$ is $\{\mathcal{F}_t^{W^{\QQ(a)}}\vee\mathcal{F}_t^L\}_{t\in[0,T]}$-adapted,
\begin{align*}
    Y_T^{\QQ(a)}|\mathcal{F}_T^{W^{\QQ(a)}}\vee\mathcal{F}_T^L\sim\text{Lognormal}\left(-\frac{1-\rho^2}{2}\int_0^Tv_sds,(1-\rho^2)\int_0^Tv_sds\right)
\end{align*}
and we have that $\EE^{\QQ(a)}[Y_T^{\QQ(a)}|\mathcal{F}_T^{W^{\QQ(a)}}\vee\mathcal{F}_T^L]=1.$ Replacing the last expression in \eqref{cond.1} we obtain that 
$\EE^{\QQ(a)}[\widehat{S}_T]=S_0\EE^{\QQ(a)}[Z_T^{\QQ(a)}],$ and, hence, we only need to check $\EE^{\QQ(a)}[Z_T^{\QQ(a)}]=1$. We will prove that Novikov's condition holds, that is, 
\begin{align}\label{condQ}
    \EE^{\QQ(a)}\left[\exp\left(\frac{\rho^2}{2}\int_0^Tv_udu\right)\right]<\infty.
\end{align}
Differently from the proof of \cite[Theorem 3.6]{onchanges}, we can not directly apply Proposition \ref{novikov} with the volatility parameters $\kappa^{(a)}$ and $\Bar{v}^{(a)}$ given in \eqref{underQ} to prove that the previous expectation is finite. One reason is that under ${\QQ(a)}$ the Brownian motion $W^{\QQ(a)}$ and the compound Hawkes process $L$ are not longer independent because 
\begin{align*}
    dW_t^{\QQ(a)}=dW_t+a\sqrt{v_t}dt,
\end{align*}
and, therefore, the dynamics of $v$ is different under ${\QQ(a)}$. In order to check $(\ref{condQ})$ we note that
\begin{align*}
    \EE^{\QQ(a)}\left[\exp\left(\frac{\rho^2}{2}\int_0^Tv_udu\right)\right]=\EE\left[\exp\left(\frac{\rho^2}{2}\int_0^Tv_udu\right)\frac{d\QQ(a)}{d\PP}\right],
\end{align*}
where 
\begin{align*}
    \frac{d\QQ(a)}{d\PP}=Y_T^{(a)}Z_T^{(a)},
\end{align*}
and $Y_T^{(a)}$ and $Z_T^{(a)}$ are given in the statement of Theorem \ref{risk}. 

Repeating the same argument as in Theorem \ref{risk}, we can write
\begin{align*}
    \EE^{\QQ(a)}\left[\exp\left(\frac{\rho^2}{2}\int_0^Tv_udu\right)\right] &=\EE\left[\exp\left(\frac{\rho^2}{2}\int_0^Tv_udu\right)Y_T^{(a)}Z_T^{(a)}\right] \\
    & = \EE\left[\EE\left[ \exp\left(\frac{\rho^2}{2}\int_0^Tv_udu\right)Y_T^{(a)}Z_T^{(a)}\Big|\mathcal{F}_T^W\vee\mathcal{F}_T^L\right]\right] \\
    & = \EE\left[\exp\left(\frac{\rho^2}{2}\int_0^Tv_udu\right)Z_T^{(a)}\EE\left[Y_T^{(a)}|\mathcal{F}_T^W\vee\mathcal{F}_T^L\right]\right] \\
    & = \EE\left[\exp\left(\frac{\rho^2}{2}\int_0^Tv_udu\right)Z_T^{(a)}\right] \\
    & = \EE\left[\exp\left(-a\int_0^T\sqrt{v_s}dW_s-\frac{1}{2}(a^2-\rho^2)\int_0^Tv_sds\right)\right],
\end{align*}
Then, adding and subtracting $a^2\int_0^Tv_sds$ in the exponential and applying Cauchy-Schwarz inequality, we obtain
\begin{align}\label{emm}
     &\EE\left[\exp\left(-a\int_0^T\sqrt{v_s}dW_s-\frac{1}{2}(a^2-\rho^2)\int_0^Tv_sds\right)\right]  = \notag \\
     &=\EE\left[\exp\left(-a\int_0^T\sqrt{v_s}dW_s-a^2\int_0^Tv_sds\right)\exp\left(\frac{1}{2}(a^2+\rho^2)\int_0^tv_sds\right)\right] \notag\\
     &\leq \EE\left[\exp\left(-2a\int_0^T\sqrt{v_s}dW_s-2a^2\int_0^Tv_sds\right)\right]^{\frac{1}{2}}\EE\left[\exp\left((a^2+\rho^2)\int_0^tv_sds\right)\right]^{\frac{1}{2}}.
\end{align}
Note that the first factor in \eqref{emm} is the expectation of a Doléans-Dade exponential. Since $|a|<\frac{\sqrt{2c_l}}{2}$ (recall the choice of $a$ in \eqref{choicea}), $2a^2<c_l$ and by Proposition \ref{novikov}, Novikov's condition holds 
\begin{align*}
\EE\left[\exp\left(2a^2\int_0^Tv_sds\right)\right]<\infty.
\end{align*}
Therefore, we have that
\begin{align*}
     \EE\left[\exp\left(-2a\int_0^T\sqrt{v_s}dW_s-2a^2\int_0^Tv_sds\right)\right]=1.
\end{align*}
For the second term in \eqref{emm} we apply again Proposition \ref{novikov}. We need $  a^2+\rho^2<c_l,$ which is true because $|a|<\sqrt{c_l-\rho^2}$. Thus
\begin{align*}
\EE\left[\exp\left((a^2+\rho^2)\int_0^tv_sds\right)\right]<\infty,
\end{align*}
and 
\begin{align*}
     \EE^{\QQ(a)}\left[\exp\left(\frac{\rho^2}{2}\int_0^Tv_udu\right)\right]<\infty.
\end{align*}
We conclude that $\EE^{\QQ(a)}[Z_T^{\QQ(a)}]=1$, $\EE^{\QQ(a)}[\widehat{S}_T]=S_0$ and $\widehat{S}$ is a $(\mathcal{F},{\QQ(a)})$-martingale. Therefore, $\mathcal{E}_m$ is a set of equivalent martingale measures.
\end{proof}


\appendix
\section{Appendix: Technical lemmas}\label{sec: appendix}
We now give the proofs that were postponed in Section 3.

\begin{lem}\label{Alode2} For $c\leq\frac{\kappa^2}{2\sigma^2}$, define $D(c):=\sqrt{\kappa^2-2\sigma^2c}$, $\Lambda(c):=\frac{2\eta c\left(e^{D(c)T}-1\right)}{D(c)-\kappa+\left(D(c)+\kappa\right)e^{D(c)T}}$
and
\begin{align*}
    c_l:=\sup\left\{c\leq\frac{\kappa^2}{2\sigma^2}: \Lambda(c)< \epsilon_J \hspace{0.3cm}\text{and}\hspace{0.3cm} M_J\left(\Lambda(c)\right)\leq\frac{\beta}{\alpha}\exp\left(\frac{\alpha}{\beta}-1\right) \right\}.
\end{align*}
Then, $0<c_l\leq\frac{\kappa^2}{2\sigma^2}$ and for $c< c_l$, 
\begin{enumerate}[(i)]
    \item The ODE
\begin{align}\label{Aode1l}
    G'(t) & =-\frac{1}{2}\sigma^2G^2(t)+\kappa G(t)-c  \\
    G(T) & = 0 \notag
\end{align}
has a unique solution in the interval $[0,T]$. The solution is strictly decreasing and it is given by 
\begin{align*}
    G(t)=\frac{2c\left(e^{D(c)(T-t)}-1\right)}{D(c)-\kappa+\left(D(c)+\kappa\right)e^{D(c)(T-t)}}.
\end{align*}
    \item The function $t\mapsto M_J(\eta G(t))$ is well defined for $t\in[0,T]$.
    \item Define $U:=\sup_{t\in [0,T]}M_J(\eta G(t))$. Then, $U=M_J(\eta G(0))$ and 
\begin{align*}
    1<U\leq \frac{\beta}{\alpha}\exp\left(\frac{\alpha}{\beta}-1\right).
\end{align*} 
\item The ODE
\begin{align}\label{Aode2l}
    H'(t) & =\beta H(t)-M_J\left(\eta G(t)\right)\exp\left(\alpha H(t)\right) +1 \\
    H(T) & = 0 \notag
\end{align}
has a unique solution in $[0,T]$.
\end{enumerate} 
\end{lem}
\begin{proof}
We first check that $c_l>0$. Since
\begin{align}\label{limit}
    \lim_{c\rightarrow0^+}\Lambda(c)=0,
\end{align}
there exist positive values of $c$ satisfying the inequality $\Lambda(c)<\epsilon_J$. Using that $\alpha<\beta$, one can check that $\frac{\beta}{\alpha}\exp\left(\frac{\alpha}{\beta}-1\right)>1$. Since the limit in \eqref{limit} holds, $M_J(0)=1$ and $M_J$ is a continuous function there exist positive values of $c$ satisfying the inequality 
$ M_J\left(\Lambda(c)\right)\leq\frac{\beta}{\alpha}\exp\left(\frac{\alpha}{\beta}-1\right).$ Therefore, $c_l>0$. From now on, let $c< c_l$. 

(i) To find the solution we can transform equation \eqref{Aode1l} to a second order linear equation with constant coefficients and then apply the standard method to solve it using that $c<\frac{\kappa^2}{2\sigma^2}$. To see that $G$ is strictly decreasing one can check that
\begin{align*}
    G'(t)=\frac{-4cD(c)^2e^{D(c)(T-t)}}{\left(D(c)-\kappa+\left(D(c)+\kappa\right)e^{D(c)(T-t)}\right)^2}<0.
\end{align*}

(ii) Since $G$ is strictly decreasing and $\eta>0$
\begin{align*}
    \sup_{t\in[0,T]}\eta G(t)=\eta G(0)=\frac{2\eta c\left(e^{D(c)T}-1\right)}{D(c)-\kappa+(D(c)+\kappa)e^{D(c)T}}=\Lambda(c).
\end{align*}
By definition of $c_l$ we have $\Lambda(c)<\epsilon_J$. Then, $\eta G(t)<\epsilon_J$ for $t\in[0,T]$ and $M_J(\eta G(t))$ is well defined for $t\in[0,T]$. 

(iii) Since $G$ is strictly decreasing and $M_J$ is strictly increasing we have $M_J(\eta G(0))\geq M_J(\eta G(t))$ for all $t\in[0,T]$. Therefore, 
$U=M_J(\eta G(0)).$ Since $\eta G(0)>\eta G(T)=0$, we have that
\begin{align*}
    U=M_J(\eta G(0))>M_J(\eta G(T))=M_J(0)=1.
\end{align*}
Moreover, by definition of $c_l$ we have
\begin{align*}
    U=M_J(\eta G(0))=M_J\left(\frac{2\eta c(e^{D(c)T}-1)}{D(c)-\kappa+(D(c)+\kappa)e^{D(c)T}}\right)=M_J(\Lambda(c))\leq\frac{\beta}{\alpha}\exp\left(\frac{\alpha}{\beta}-1\right).
\end{align*}

(iv) Let us make the change of variables $h(t):=H(T-t)$. Then, the ODE in \eqref{Aode2l} is transformed to 
\begin{align}\label{Aodeh}
    h'(t) & = f(t,h(t))= M_J\left(\eta G(T-t)\right)\exp\left(\alpha h(t)\right)-\beta h(t)-1 \\
    h(0) & = 0 \notag
\end{align}
where $f(t,x):=M_J(\eta G(T-t))\exp\left(\alpha x\right)-\beta x-1$. 
Note that for $t\in[0,T]$ and $x\in\R$
\begin{align}\label{Aboundedh}
    f_m(x):=-\beta x-1 \leq f(t,x)\leq U\exp\left(\alpha x\right)-\beta x-1=:f_M(x).
\end{align}
First, we focus on the ODE
\begin{align}\label{Aodea}
    h_M'(t) & = f_M(h_M(t))= U\exp\left(\alpha h_M(t)\right)-\beta h_M(t)-1 \\
    h_M(0) & = 0 \notag
\end{align}
Since $f_M$ is continuously differentiable in $\RR$, it is Lipschitz continuous on bounded intervals and there exists a unique local solution for every initial condition, see \cite[Chapter II, Theorem 1.1]{hartman}. We want $f_M(0)>0$ and the existence of $x_p>0$ such that $f_M(x_p)\leq0$, that would imply the existence of a stable equilibrium point in the interval $(0,x_p)$. Therefore, the solution of \eqref{Aodea} would be well defined on $[0,\infty)$ and $h_M(t)<x_p$ for all $t\in[0,\infty)$.

Note that $f_M(0)=U-1>0$ as it was seen in the previous part.

The minimum of $f_M$ is achieved at $x_\text{min}=\frac{1}{\alpha}\ln\left(\frac{\beta}{\alpha U}\right)$. Note that $x_\text{min}>0$ if and only if $U<\frac{\beta}{\alpha}$. Then, 
\begin{align*}
    f_M(x_\text{min})=\frac{\beta}{\alpha}\left(1-\ln\left(\frac{\beta}{\alpha U}\right)\right)-1 \leq 0 
\end{align*}
if, and only if $U\leq\frac{\beta}{\alpha}\exp\left(\frac{\alpha}{\beta}-1\right),$ which is fulfilled by (iii). Moreover, note that
$\frac{\beta}{\alpha}\exp\left(\frac{\alpha}{\beta}-1\right)<\frac{\beta}{\alpha}$ and then $x_\text{min}>0$.

We conclude that the point we were searching is $x_p=x_\text{min}$. This guarantees that $h_M$ is well defined on $[0,\infty)$ and $h_M(t)<x_p$ for all $t\in[0,T]$. 

Now we focus on the ODE
\begin{align}\label{Aodeb}
    h_m'(t) & = f_m(h_m(t))= -\beta h_m(t)-1 \\
    h_m(0) & = 0 \notag
\end{align}
The solution is given by $ h_m(t)=\frac{e^{-\beta t}-1}{\beta}.$ The function $h_m$ is well defined on $[0,\infty)$ and $h_m(t)\geq\frac{-1}{\beta}$ for all $t\in[0,\infty)$. 

Consider again to the ODE in \eqref{Aodeh} and recall that 
\begin{align*}
    f(t,x)=M_J\left(\eta G(T-t)\right)\exp\left(\alpha x)\right)-\beta x-1.
\end{align*}
Define the open interval $V:=(-\frac{1}{\beta}-1,x_p+1)$. Note that $f\colon[0,T] \times V\to\R$ is a continuous function, Lipschitz in $x$ because the exponential function is Lipschitz on bounded intervals. For $(t,x_1),(t,x_2)\in[0,T]\times V$ we have
\begin{align*}
    |f(t,x_1)-f(t,x_2)|&\leq |M_J(\eta G(T-t))||\exp\left(\alpha x_1\right)-\exp\left(\alpha x_2\right)|+\beta|x_1-x_2| \\
    & \leq U|\exp\left(\alpha x_1\right)-\exp\left(\alpha x_2\right)|+\beta|x_1-x_2| \\
    &\leq K_1|x_1-x_2|,
\end{align*}
for some constant $K_1>0$. Then, by the Picard-Lindelöf theorem (see \cite[Chapter II, Theorem 1.1]{hartman}) there is a unique solution $h\colon I\to V$ for some interval $I\subset[0,T]$. Moreover, by \cite[Chapter II, Theorem 3.1]{hartman} only two cases are possible:

1) $I=[0,T]$. In that case, there is nothing more to prove.

2) $I=[0,\epsilon)$ with $\epsilon\leq T$ and
\begin{align}\label{Aboundary}
    \lim_{t\rightarrow\epsilon^-}h(t)\in\left\{-\frac{1}{\beta}-1,x_p+1\right\}.
\end{align}
That is, $h$ approaches the boundary of $V$ when $t$ approaches $\epsilon$. 
However, as a consequence of \eqref{Aboundedh}, we will prove that $h_m(t)\leq h(t)\leq h_M(t),$ for all $t\in[0,\epsilon)$.

First, we prove that $h_m(t)\leq h(t)$ for all $t\in[0,\epsilon)$. Define the function $g(t):=h_m(t)-h(t)$ for $t\in[0,\epsilon)$. Note that $g(0)=0$ and we want to prove that $g(t)\leq 0$ for all $t\in[0,\epsilon)$.
Assume there exists $s\in(0,\epsilon)$ such that $g(s)>0$. Since $g$ is continuous and $g(0)=0$, there exists $r\in[0,s)$ with $g(r)=0$ and $g(t)>0$ for $t\in(r,s]$. Now, for $t\in[r,s]$ we have
\begin{align*}
    g'(t)&=h_m'(t)-h'(t) \\
    &=f_m(h_m(t))-f(t,h(t)) \\
    &\leq f(t,h_m(t))-f(t,h(t)) \\
    &= M_J(\eta G(T-t))\left(\exp\left(\alpha h_m(t)\right)-\exp\left(\alpha h(t)\right)\right)-\beta\left(h_m(t)-h(t)\right)  \\
    &\leq U\left(\exp\left(\alpha h_m(t)\right)-\exp\left(\alpha h(t)\right)\right)-\beta\left(h_m(t)-h(t)\right)  \\
    &\leq K_2 |h_m(t)-h(t)| = K_2 |g(t)| = K_2 g(t),
 \end{align*}
for some constant $K_2>0$, where in the last inequality we have used that since $h_m$ and $h$ are continuous, they are bounded in $[r,s]$ and we can use the Lipschitz property because the exponential function is Lipschitz on bounded intervals. Applying Gronwall's inequality we have that $g(s)\leq g(r)e^{K_2(s-r)}=0,$ which is a contradiction. We conclude that $h_m(t)\leq h(t)$ for all $t\in[0,\epsilon)$. A similar same argument can be employed to prove that $h(t)\leq h_M(t)$ for all $t\in[0,\epsilon)$.

For $t\in[0,\epsilon)$ we have 
\begin{align*}
    h_m(t)\leq h(t)\leq h_M(t) \implies \frac{-1}{\beta}\leq h(t)\leq x_p \implies \frac{-1}{\beta}\leq\lim_{t\rightarrow\epsilon^-}h(t)\leq x_p.
\end{align*}
This contradicts \eqref{Aboundary} and we conclude that the only possible situation is that $h$ is well defined on $[0,T]$. 
\end{proof}

\begin{cor} Define $c_s$ by 
\label{Aexplicit}
\begin{align}\label{defcs}
    c_s:=\min\left\{\frac{\kappa\epsilon_J}{2\eta},\frac{\kappa }{2\eta}M_J^{-1}\left(\frac{\beta}{\alpha}\exp\left(\frac{\alpha}{\beta}-1\right)\right),\frac{\kappa^2}{2\sigma^2}\right\}.
\end{align}
Then, $0<c_s<c_l$.
\end{cor}
\begin{proof}
We first check that $c_s>0$. The function $M_J\colon(-\infty,\epsilon_J)\to(0,\infty)$ is well defined and strictly increasing. Therefore, $M_J^{-1}\colon(0,\infty)\to(-\infty,\epsilon_J)$ is also a well defined function and it is strictly increasing. 

Since $\frac{\beta}{\alpha}\exp\left(\frac{\alpha}{\beta}-1\right)>1$ and $M_J(0)=1$ we have $M_J^{-1}\left(\frac{\beta}{\alpha}\exp\left(\frac{\alpha}{\beta}-1\right)\right)>0,$ and we can conclude that $c_s>0$. 

Recall that $0<
c_l\leq\frac{\kappa^2}{2\sigma^2}$ and its definition 
\begin{align}\label{defcl}
    c_l=\sup\left\{c\leq\frac{\kappa^2}{2\sigma^2}: \Lambda(c)< \epsilon_J \hspace{0.3cm}\text{and}\hspace{0.3cm} M_J\left(\Lambda(c)\right)\leq\frac{\beta}{\alpha}\exp\left(\frac{\alpha}{\beta}-1\right) \right\},
\end{align}
where $D(c)=\sqrt{\kappa^2-2\sigma^2c}$ and $\Lambda(c)=\frac{2\eta c\left(e^{D(c)T}-1\right)}{D(c)-\kappa+\left(D(c)+\kappa\right)e^{D(c)T}}$. 

To prove that $c_s<c_l$ we check that $\Lambda(c_s)<\epsilon_J$ and $M_J\left(\Lambda(c_s)\right)\leq\frac{\beta}{\alpha}\exp\left(\frac{\alpha}{\beta}-1\right)$. The following inequality holds
\begin{align}\label{Atrick}
     \Lambda(c)=\frac{2\eta c\left(e^{\sqrt{\kappa^2-2\sigma^2c}T}-1\right)}{\sqrt{\kappa^2-2\sigma^2c}-\kappa+\left(\sqrt{\kappa^2-2\sigma^2c}+\kappa\right)e^{\sqrt{\kappa^2-2\sigma^2c}T}}< \frac{2\eta c}{\kappa}.
\end{align}
Actually
\begin{align*}
\Lambda(c)=\frac{2\eta c\left(e^{\sqrt{\kappa^2-2\sigma^2c}T}-1\right)}{\sqrt{\kappa^2-2\sigma^2c}-\kappa+\left(\sqrt{\kappa^2-2\sigma^2c}+\kappa\right)e^{\sqrt{\kappa^2-2\sigma^2c}T}}< \frac{2\eta c}{\kappa} \\
\iff 
    \kappa\left(e^{\sqrt{\kappa^2-2\sigma^2c}T}-1\right)<\sqrt{\kappa^2-2\sigma^2c}-\kappa+\left(\sqrt{\kappa^2-2\sigma^2c}+\kappa\right)e^{\sqrt{\kappa^2-2\sigma^2c}T} \\
    \iff 0< \sqrt{\kappa^2-2\sigma^2c}\left(1+e^{\sqrt{\kappa^2-2\sigma^2c}T}\right).
\end{align*}
Now, by definition of $c_s$ we have $ \Lambda(c_s)<\frac{2\eta c_s}{\kappa}\leq\epsilon_J,$ and
\begin{align*}
    M_J\left(\Lambda(c_s)\right)<M_J\left(\frac{2\eta c_s}{\kappa}\right)\leq M_J\left(M_J^{-1}\left(\frac{\beta}{\alpha}\exp\left(\frac{\alpha}{\beta}-1\right)\right)\right)=\frac{\beta}{\alpha}\exp\left(\frac{\alpha}{\beta}-1\right).
\end{align*}
We conclude that $c_s<c_l$. Note that the inequality in $c_s<c_l$ is strict because the inequality in \eqref{Atrick} is strict.
\end{proof}
\begin{exam}\label{exA}
Some examples:
\begin{enumerate}[(i)]
    \item If $J_1\sim\text{Exponential}(\lambda)$, then 
\begin{align*}
    c_s=\min\left\{\frac{\kappa\lambda}{2\eta}\left(1-\frac{\alpha}{\beta}\exp\left(1-\frac{\alpha}{\beta}\right)\right),\frac{\kappa^2}{2\sigma^2}\right\}.
\end{align*}
\item If $J_1\sim\text{Gamma}(\mu,\lambda)$ with $\mu,\lambda>0$ as the shape and the rate, respectively. Then
\begin{align*}
    c_s=\min\left\{\frac{\kappa\lambda }{2\eta}\left(1-\frac{1}{\left(\frac{\beta}{\alpha}\exp\left(\frac{\alpha}{\beta}-1\right)\right)^{1/\mu}}\right),\frac{\kappa^2}{2\sigma^2}\right\}.
\end{align*}
\item If $J_1=j>0$, then
\begin{align*}
    c_s=\min\left\{\frac{\kappa}{2\eta j}\left(\ln\left(\frac{\beta}{\alpha}\right)+\frac{\alpha}{\beta}-1\right),\frac{\kappa^2}{2\sigma^2}\right\}.
\end{align*}
\end{enumerate}
\end{exam}
\begin{proof}
(i) The moment generating function is given by $M_J(t)=\frac{\lambda}{\lambda-t}, t<\lambda.$ Hence, with the notation of Assumption \ref{as}, $\epsilon_J=\lambda$ and the inverse of the $M_J$ is given by $M_J^{-1}(t)=\lambda\left(1-\frac{1}{t}\right), t>0$. Then, applying Corollary \ref{Aexplicit} we have the following expression of $c_s$
\begin{align*}
    c_s=\min\left\{\frac{\kappa\lambda}{2\eta},\frac{\kappa\lambda }{2\eta}\left(1-\frac{1}{\frac{\beta}{\alpha}\exp\left(\frac{\alpha}{\beta}-1\right)}\right),\frac{\kappa^2}{2\sigma^2}\right\}.
\end{align*}
Note that since $\frac{\beta}{\alpha}\exp\left(\frac{\alpha}{\beta}-1\right)>1$, we have
\begin{align*}
    0<1-\frac{1}{\frac{\beta}{\alpha}\exp\left(\frac{\alpha}{\beta}-1\right)}<1.
\end{align*}
We conclude that 
\begin{align*}
    c_s&=\min\left\{\frac{\kappa\lambda}{2\eta},\frac{\kappa\lambda }{2\eta}\left(1-\frac{1}{\frac{\beta}{\alpha}\exp\left(\frac{\alpha}{\beta}-1\right)}\right),\frac{\kappa^2}{2\sigma^2}\right\} \\
    & =\min\left\{\frac{\kappa\lambda}{2\eta}\left(1-\frac{\alpha}{\beta}\exp\left(1-\frac{\alpha}{\beta}\right)\right),\frac{\kappa^2}{2\sigma^2}\right\}.
\end{align*}

\noindent
(ii) The moment generating function is given by $M_J(t)=\left(1-\frac{t}{\lambda}\right)^{-\mu}, t<\lambda.$ Thus, $\epsilon_J=\lambda$ and the inverse of $M_J$ is given by $M_J^{-1}(t)=\lambda\left(1-t^{-1/\mu}\right), t>0.$ Then, applying Corollary \ref{Aexplicit} we have the following expression of $c_s$
\begin{align*}
    c_s=\min\left\{\frac{\kappa\lambda}{2\eta},\frac{\kappa\lambda }{2\eta}\left(1-\left(\frac{\beta}{\alpha}\exp\left(\frac{\alpha}{\beta}-1\right)\right)^{-1/\mu}\right),\frac{\kappa^2}{2\sigma^2}\right\}.
\end{align*}
Note that since $\frac{\beta}{\alpha}\exp\left(\frac{\alpha}{\beta}-1\right)>1$, we have
\begin{align*}
    0<1-\frac{1}{\left(\frac{\beta}{\alpha}\exp\left(\frac{\alpha}{\beta}-1\right)\right)^{1/\mu}}<1.
\end{align*}
We conclude that 
\begin{align*}
    c_s=\min\left\{\frac{\kappa\lambda }{2\eta}\left(1-\left(\frac{\beta}{\alpha}\exp\left(\frac{\alpha}{\beta}-1\right)\right)^{-1/\mu}\right),\frac{\kappa^2}{2\sigma^2}\right\}.
\end{align*}

\noindent
(iii) The moment generating function is given by $M_J(t)=e^{tj}$, $t\in\RR$. Thus, $\epsilon_J=\infty$ and the inverse of $M_J$ is given by $M_J^{-1}(t)=\ln(t)/j, t>0$. Then, applying Corollary \ref{Aexplicit} we conclude that 
\begin{align*}
    c_s=\min\left\{\frac{\kappa}{2\eta j}\left(\ln\left(\frac{\beta}{\alpha}\right)+\frac{\alpha}{\beta}-1\right),\frac{\kappa^2}{2\sigma^2}\right\}.
\end{align*}
\end{proof}


\printbibliography

\end{document}